\documentclass[12pt,a4paper]{amsart}
\usepackage{amsfonts,amssymb,amscd,amsmath,enumerate,verbatim}
\usepackage{amsmath,amsthm,amssymb}
\usepackage{ulem,ascmac} 
\usepackage[dvips]{graphicx}
\usepackage{fancybox}

\usepackage{url}

\theoremstyle{plain}
\newtheorem{thm}{Theorem}[section]
\newtheorem{prop}[thm]{Proposition}
\newtheorem{cor}[thm]{Corollary}

\newtheorem{conj}[thm]{Conjecture}

\theoremstyle{definition}

\newtheorem{obs}[thm]{Observation}

\newcommand{\Rc}{{\mathcal{R}}}
\newcommand{\Lc}{{\mathcal{L}}}
\newcommand{\Tc}{{\mathcal{T}}}

\newcommand{\Hc}{{\mathcal{H}}}
\newcommand{\Sc}{{\mathbb{S}}}
\newcommand{\Zc}{{\mathcal{Z}}}

\newcommand{\rleft}{\mathopen{}\mathclose\bgroup\left}
\newcommand{\rright}{\aftergroup\egroup\right}

\title[Ratio of the numbers of odd/even cycles in outerplanar graphs]{The ratio of the numbers of odd and even cycles in outerplanar graphs}

\author[A Higashitani \and N. Matsumoto]{Akihiro Higashitani \and Naoki Matsumoto} 
\address[A Higashitani]{Department of Pure and Applied Mathematics, Graduate School of Information Science and Technology, Osaka University, Suita, Osaka 565-0871, Japan}
\email{higashitani@ist.osaka-u.ac.jp}
\address[N. Matsumoto]{Research Institute for Digital Media and Content, Keio University, Yokohama, Kanagawa 232-0062, Japan}
\email{naoki.matsumo10@gmail.com}
\thanks{}

\subjclass[2020]{Primary: 05C30; Secondary: 05C31}
\keywords{Enumeration, Outerplanar graph, Cycle, Subtree, Claw-free}
\thanks{The first named author is supported by JSPS Grant-in-Aid for Scientific Research (C) 20K03513. 
The second named author is supported by JSPS Grant-in-Aid for Early-Career Scientists 19K14583.}

\begin{document}

\maketitle

\begin{abstract}
In this paper, we investigate the ratio of the numbers of odd and even cycles in outerplanar graphs.
We verify that the ratio generally diverges to infinity as the order of a graph diverges to infinity.
We also give sharp estimations of the ratio for several classes of outerplanar graphs,
and obtain a constant upper bound of the ratio for some of them.
Furthermore, 
we consider similar problems in graphs with some pairs of forbidden subgraphs/minors,
and propose a challenging problem concerning claw-free graphs.
\end{abstract}

\section{Introduction}
How different is the number of odd cycles and that of even cycles in a graph?
(The number of subgraphs means that of {\it distinct} subgraphs; see Subsection~\ref{subsec:1.3} for details.)
This is a very natural and fundamental question in graph theory,
but, as far as we know,
there is no serious study for this question.
Thus, we investigate the ratio of the number of odd cycles and that of even cycles in graphs,
in particular, in outerplanar graphs.

In the introduction,
we first provide a short survey for the study on graph polynomials 
some of which can be applied to evaluate the number of odd cycles and that of even cycles in outerplanar graphs.
Next we summarize main previous studies on the number of cycles in graphs,
and then we describe our results and the organization of this paper.
For fundamental terminologies and notations undefined in this paper, we refer the reader to~\cite{BMGT}.

\subsection{Graph polynomial}

The {\it chromatic polynomial} is the most classical graph polynomial, introduced by Birkhoff (cf.~\cite{BL1946}),
which counts the number of proper colorings of a graph with a given number of colors.
This polynomial can count the number of acyclic orientations of a graph 
by assigning $-1$ to the (unique) variable of the polynomial.
Tutte~\cite{Tutte1954} developed a more general graph polynomial, namely the {\it Tutte polynomial}.
This polynomial can count forests and spanning subgraphs (or spanning forests) 
by suitably setting values of two variables,
and specializes to the chromatic polynomial by assigning some constant to one of two variables.
For counting of spanning subgraphs of a graph,
Farrell~\cite{Farrell1979} introduced a {\it family polynomial} (or {\it $F$-polynomial}\/), 
which counts the number of spanning subgraphs of a graph with each component belonging to the family.

Jamison~\cite{Jami1983} developed a {\it subtree polynomial} which counts all subtrees of a tree,
and well investigated the coefficients of the polynomial~\cite{Jami1984,Jami1987,Jami1990}.
In particular, he discovered the following interesting fact on the difference between the number of odd subtrees and that of even subtrees,
which can be used to 
evaluate the ratio of the number of odd cycles and that of even cycles in some family of outerplanar graphs,
where a graph is {\it odd} (resp.~{\it even}) if the order of the graph is odd (resp.~even).

\begin{thm}[\cite{Jami1987}]\label{thm:subtrees}
For any tree $T$, the number of odd subtrees minus that of even subtrees in $T$ 
is equal to the independence number of $T$. 
\end{thm}

The research has been continuing on subtree polynomials; for example, see~\cite{BM2020,VW2010}.
A few years ago, Dod et al.~\cite{DKPT2015} introduced a {\it bipartition polynomial} 
which is a common generalization of several polynomials, namely,
the domination polynomial, the Ising polynomial, the matching polynomial and the cut polynomial.
Note that each cycle in a planar graph $G$ drawn on the plane without edge crossings
corresponds to a cut (i.e., a set of edges which makes a graph disconnected) in the dual graph of $G$,
where the {\it dual graph} of a graph $G$ drawn on the plane is a plane graph
which has a vertex for each face of $G$ and an edge for each pair of faces in $G$ sharing an edge of $G$.
Thus, the cut polynomial can evaluate the number of cycles in a planar graph, 
but it is hard in general to exactly count the number of cycles of a given planar graph
even if we use this polynomial; in fact, this counting problem is $\#$P-complete~\cite{Vija2009}.
For other polynomials and related topics, see a book~\cite{polynomial_book}.
(We can also define the cycle polynomial whose coefficients are the number of cycles. 
However, this is just an analogy of other polynomials, 
and hence, there seems no serious study on this polynomial.\footnote{
For several graphs, the cycle polynomial is determined; 
see \url{https://mathworld.wolfram.com/CyclePolynomial.html}
})

\subsection{The number of cycles}

For a graph $G$, the classical estimation of the number of cycles $\nu(G)$ was given as follows:
$$\mu(G) \leq \nu(G) \leq 2^{\mu(G)} - 1,$$
where $\mu(G)$ is the circuit rank (or the cyclomatic number) of $G$.
Volkmann~\cite{Volk1996} gave another lower bound to $\nu(G)$ using the minimum degree of $G$,
and some authors study graphs $G$ with $ \nu(G) = 2^{\mu(G)} - 1$;
for example, see~\cite{AGT2016,RS2005}.
Counting cycles in graphs has probably begun in 1960s by several groups, e.g.,~\cite{CG1966,HM1971}.
Khomenko and Golovko~\cite{KG1972} gave a formula counting 
the number of cycles of a given length using the adjacency matrix.
However, described as in the previous subsection,
since counting cycles in a given graph is hard in general, 
there are many studies on the number of cycles for particular graph classes; 
for example, see~\cite{AlB2018}.

It is deeply and widely studied to count the number of hamiltonian cycles.
It is well known that determining whether a graph has a hamiltonian cycle is NP-hard~\cite{Karp1972},
and hence, many authors count the number of hamiltonian cycles in prescribed graph classes, e.g., 
bipartite graphs~\cite{Thomassen1996}, regular graphs~\cite{Hay2017} and planar triangulations~\cite{BSV2018}.
For related studies and other topics, see surveys~\cite{Gould_survey, Oz_survey}, 
and for the directed version of such problems, see another survey~\cite{KO_survey}.
Moreover, we refer the reader to a more recent paper~\cite{AAT2020}
which investigates the number of (hamiltonian) cycles in planar graphs with prescribed connectivity.

Despite of those several kinds of results on the number of cycles, 
there seems no result on the difference between the number of odd cycles and that of even cycles as far as we know.

\subsection{Contribution}\label{subsec:1.3}

We first introduce terminologies and notations to mention our results. 
All graphs considered in this paper are finite simple undirected graphs. 
For a graph $G$, we denote by $V(G), E(G)$ and $F(G)$ the set of vertices, edges and faces, respectively. 
Note that we define $F(G)$ only if $G$ is embedded on some surface. 
In this paper, the number of subgraphs means the number of distinct subgraphs,
where two subgraphs $H_1$ and $H_2$ of $G$ are {\it distinct} if $V(H_1) \neq V(H_2)$ or $E(H_1) \neq E(H_2)$.
Let $c_o(G)$ and $c_e(G)$ denote the number of odd cycles and that of even cycles, respectively.

Let $G$ be an outerplanar graph embedded on the plane
so that all vertices lie on the boundary walk of the infinite face,
where a {\it finite} (resp.~{\it infinite}) face of $G$ is 
a face of $G$ with its boundary walk bounding a finite (resp.~infinite) region.
In what follows, an outerplanar graph means one embedded on the plane.
Note that there is only one infinite face for any outerplanar graph.
The boundary walk of the infinite face of $G$ is denoted by $\partial G$
and an edge $e$ of $G$ is {\it diagonal}\/ if $e$ does not belong to $\partial G$.
We say that $\partial G$ is {\it odd} (resp. {\it even}) if the length of $\partial G$ is odd (resp. even). 
When $\partial G$ is an even cycle, we always color $\partial G$ by two colors, black and white. 
In this case, an {\it odd} (resp.~{\it even}) {\it chord} of $G$
is a diagonal of $G$ joining two vertices with the same color (resp.~distinct colors).
When $\partial G$ is a cycle,
the {\it dual tree} of $G$, denoted by $T_G$, is the graph obtained from the dual graph of $G$ 
by deleting the vertex corresponding to the infinite face of $G$.

In this paper, for an outerplanar graph $G$ where $\partial G$ is a cycle, i.e., $G$ is 2-connected, 
we have the following results on $c_o(G)/c_e(G)$ and $c_e(G)/c_o(G)$:

\begin{itemize}
\setlength{\itemsep}{3mm} 

\item
Both $c_o(G)/c_e(G)$ and $c_e(G)/c_o(G)$ diverge to infinity as $|V(G)| \to \infty$ in general,
not depending on the parity of $\partial G$ (Theorem~\ref{thm:exist}).

\item 
If the size of each finite face of an outerplanar graph $G$ is odd and $|F(G)| \geq 3$,
then $\frac{|F(G)|-1}{2c_e(G)} + 1 \leq c_o(G)/c_e(G) \leq \frac{|F(G)|-2}{c_e(G)} + 1$ 
(Theorem~\ref{thm:oddfaces}).

\item If $\partial G$ is even,
then $c_o(G)/c_e(G) \leq k$, where $k$ is the number of odd chords, 
and this bound is sharp (Theorem~\ref{thm:oddchords}).

\item If $\partial G$ is even and the dual tree $T_G$ is a path,
then $c_o(G)/c_e(G) \leq 2$, and this bound is sharp (Proposition~\ref{prop:dualpath}).

\item If the dual tree $T_G$ is a star and there is at least one odd face corresponding to a leaf of $T_G$,
then both $c_o(G)/c_e(G)$ and $c_e(G)/c_o(G)$ converge to $1$ as $|F(G)| \to \infty$
(Proposition~\ref{prop:dualstar}).
\end{itemize}

Note that if a graph $G$ is bipartite, then $c_o(G)=0$ always holds. 
Thus, we deal with only non-bipartite graphs in what follows.
Moreover, if an outerplanar graph $G$ is not 2-connected,
then we can obtain similar results by applying the above results to each block of $G$ (see Section~\ref{sec:remarks}). 

In Section~\ref{sec:forbidden},
we consider conditions for a graph $G$ concerning forbidden subgraphs/minors 
that $c_o(G)/c_e(G)$ or $c_e(G)/c_o(G)$ is bounded by some constant.
Recall that every outerplanar graph 
is characterized as a $K_{2,3}$-minor-free and $K_4$-minor-free graph~\cite[Exercise 10.5.12]{BMGT}.
Our first result (Theorem~\ref{thm:exist}) implies that 
the $K_{2,3}$-minor-free and $K_4$-minor-free condition is not sufficient
to bound $c_o(G)/c_e(G)$ or $c_e(G)/c_o(G)$ by a constant.
Moreover, by application of Wagner's proof~\cite{Wagner},
every 2-connected graph is $K_{2,3}$-minor-free if and only if 
it is either isomorphic to $K_4$ or outerplanar (cf.~\cite{EMOT2016}).
Hence, replacing $K_{2,3}$-minor-free with $K_{1,3}$-free (or, popularly, claw-free),
we show that every $K_{1,3}$-free and $K_4$-minor-free graph is outerplanar,
and we completely characterize the structure of such outerplanar graphs (Theorem~\ref{thm:forbid_claw_k4}).
As a corollary, for any $K_{1,3}$-free and $K_4$-minor-free graph $G$,
both $c_o(G)/c_e(G)$ and $c_e(G)/c_o(G)$ are bounded by a constant unless $G$ is a cycle (Corollary~\ref{cor:claw_k4}).

\subsection{Organization of the paper}

In the next section,
we construct 2-connected outerplanar graphs $G$ such that $c_o(G)/c_e(G)$ or $c_e(G)/c_o(G)$ 
diverges to infinity.
In Section~\ref{sec:bounds},
we show upper/lower bounds of the ratio of the number of odd cycles and that of even cycles
in 2-connected outerplanar graphs.
In Section~\ref{sec:forbidden}, 
we consider the ratio of the number of odd cycles and that of even cycles in graphs 
characterized by forbidden subgraphs/minors regarding to outerplanar graphs.
In the final section, we give concluding remarks and future perspectives.

\section{Outerplanar graphs with many odd/even cycles}\label{sec:graphs}

We first introduce two particular 2-connected outerplanar graphs $G$ 
such that $\partial G$ is odd.

\begin{description}
\item[The graph $\Rc_k$]
Given an odd number $k$ with $k \geq 3$, prepare a cycle $C_k = v_0v_1v_2 \dots v_{k-1}v_0$ of length $k$. 
Prepare $k$ copies of the 4-cycle denoted by $D_0,D_1,\dots,D_{k-1}$.
For each $i \in \{0,1,\dots,k-1\}$,
identify an edge of $D_i$ and $v_{i}v_{i+1}$ where the subscripts are modulo~$k$.
The resulting 2-connected outerplanar graph is denoted by $\Rc_k$; see Figure~\ref{fig:gear}.

\smallskip

\item[The graph $\Lc_m$]
Given an odd number $m$ with $m \geq 3$, prepare a path $P_m = v_0v_1v_2 \dots v_{m-1}$ of order $m$. 
For each $i \in \{0,\dots,(m-3)/2\}$, join $v_i$ and $v_{m-1-i}$.
The resulting 2-connected outerplanar graph is denoted by $\Lc_{(m-1)/2}$; see Figure~\ref{fig:pencil}.

\end{description}

\begin{figure}[htb]
\begin{minipage}{0.5\hsize}
\centering
\unitlength 0.1in
\begin{picture}( 21.0000, 19.0000)(  9.5000,-24.5000)
%
\special{pn 8}%
\special{sh 1.000}%
\special{ar 1800 1000 50 50  0.0000000 6.2831853}%
%
\special{pn 8}%
\special{sh 1.000}%
\special{ar 2200 1000 50 50  0.0000000 6.2831853}%
%
\special{pn 8}%
\special{sh 1.000}%
\special{ar 1400 1400 50 50  0.0000000 6.2831853}%
%
\special{pn 8}%
\special{sh 1.000}%
\special{ar 1800 600 50 50  0.0000000 6.2831853}%
%
\special{pn 8}%
\special{sh 1.000}%
\special{ar 2200 600 50 50  0.0000000 6.2831853}%
%
\special{pn 8}%
\special{sh 1.000}%
\special{ar 1400 1800 50 50  0.0000000 6.2831853}%
%
\special{pn 8}%
\special{sh 1.000}%
\special{ar 2000 2200 50 50  0.0000000 6.2831853}%
%
\special{pn 8}%
\special{sh 1.000}%
\special{ar 2600 1800 50 50  0.0000000 6.2831853}%
%
\special{pn 8}%
\special{sh 1.000}%
\special{ar 2600 1400 50 50  0.0000000 6.2831853}%
%
\special{pn 8}%
\special{sh 1.000}%
\special{ar 1400 800 50 50  0.0000000 6.2831853}%
%
\special{pn 8}%
\special{sh 1.000}%
\special{ar 1200 1000 50 50  0.0000000 6.2831853}%
%
\special{pn 8}%
\special{sh 1.000}%
\special{ar 1000 1400 50 50  0.0000000 6.2831853}%
%
\special{pn 8}%
\special{sh 1.000}%
\special{ar 1000 1800 50 50  0.0000000 6.2831853}%
%
\special{pn 8}%
\special{sh 1.000}%
\special{ar 3000 1800 50 50  0.0000000 6.2831853}%
%
\special{pn 8}%
\special{sh 1.000}%
\special{ar 3000 1400 50 50  0.0000000 6.2831853}%
%
\special{pn 8}%
\special{sh 1.000}%
\special{ar 2800 1000 50 50  0.0000000 6.2831853}%
%
\special{pn 8}%
\special{sh 1.000}%
\special{ar 2600 800 50 50  0.0000000 6.2831853}%
%
\special{pn 8}%
\special{sh 1.000}%
\special{ar 1600 2400 50 50  0.0000000 6.2831853}%
%
\special{pn 8}%
\special{sh 1.000}%
\special{ar 2400 2400 50 50  0.0000000 6.2831853}%
%
\special{pn 8}%
\special{sh 1.000}%
\special{ar 2600 2200 50 50  0.0000000 6.2831853}%
%
\special{pn 8}%
\special{sh 1.000}%
\special{ar 1400 2200 50 50  0.0000000 6.2831853}%
%
\special{pn 8}%
\special{pa 2600 1800}%
\special{pa 2000 2200}%
\special{fp}%
%
\special{pn 8}%
\special{pa 2000 2200}%
\special{pa 1400 1800}%
\special{fp}%
%
\special{pn 8}%
\special{pa 1400 1800}%
\special{pa 1400 1400}%
\special{fp}%
%
\special{pn 8}%
\special{pa 1400 1400}%
\special{pa 1800 1000}%
\special{fp}%
%
\special{pn 8}%
\special{pa 1800 1000}%
\special{pa 2200 1000}%
\special{fp}%
%
\special{pn 8}%
\special{pa 2200 1000}%
\special{pa 2600 1400}%
\special{fp}%
%
\special{pn 8}%
\special{pa 2600 1400}%
\special{pa 2600 1800}%
\special{fp}%
%
\special{pn 8}%
\special{pa 2600 1400}%
\special{pa 2800 1000}%
\special{fp}%
%
\special{pn 8}%
\special{pa 2800 1000}%
\special{pa 2600 800}%
\special{fp}%
%
\special{pn 8}%
\special{pa 2600 800}%
\special{pa 2200 1000}%
\special{fp}%
%
\special{pn 8}%
\special{pa 2200 1000}%
\special{pa 2200 600}%
\special{fp}%
%
\special{pn 8}%
\special{pa 2200 600}%
\special{pa 1800 600}%
\special{fp}%
%
\special{pn 8}%
\special{pa 1800 600}%
\special{pa 1800 1000}%
\special{fp}%
%
\special{pn 8}%
\special{pa 1800 1000}%
\special{pa 1400 800}%
\special{fp}%
%
\special{pn 8}%
\special{pa 1400 800}%
\special{pa 1200 1000}%
\special{fp}%
%
\special{pn 8}%
\special{pa 1200 1000}%
\special{pa 1400 1400}%
\special{fp}%
%
\special{pn 8}%
\special{pa 1400 1400}%
\special{pa 1000 1400}%
\special{fp}%
%
\special{pn 8}%
\special{pa 1000 1400}%
\special{pa 1000 1800}%
\special{fp}%
%
\special{pn 8}%
\special{pa 1000 1800}%
\special{pa 1400 1800}%
\special{fp}%
%
\special{pn 8}%
\special{pa 1400 1800}%
\special{pa 1400 2200}%
\special{fp}%
%
\special{pn 8}%
\special{pa 1400 2200}%
\special{pa 1600 2400}%
\special{fp}%
%
\special{pn 8}%
\special{pa 1600 2400}%
\special{pa 2000 2200}%
\special{fp}%
%
\special{pn 8}%
\special{pa 2000 2200}%
\special{pa 2400 2400}%
\special{fp}%
%
\special{pn 8}%
\special{pa 2400 2400}%
\special{pa 2600 2200}%
\special{fp}%
%
\special{pn 8}%
\special{pa 2600 2200}%
\special{pa 2600 1800}%
\special{fp}%
%
\special{pn 8}%
\special{pa 2600 1800}%
\special{pa 3000 1800}%
\special{fp}%
%
\special{pn 8}%
\special{pa 3000 1800}%
\special{pa 3000 1400}%
\special{fp}%
%
\special{pn 8}%
\special{pa 3000 1400}%
\special{pa 2600 1400}%
\special{fp}%
\put(20.0000,-20.6000){\makebox(0,0){$v_0$}}%
\put(15.8000,-18.0000){\makebox(0,0){$v_1$}}%
\put(15.8000,-14.0000){\makebox(0,0){$v_2$}}%
\put(24.3000,-14.0000){\makebox(0,0){$v_5$}}%
\put(24.3000,-18.0000){\makebox(0,0){$v_6$}}%
\put(22.0000,-11.4000){\makebox(0,0){$v_4$}}%
\put(18.0000,-11.4000){\makebox(0,0){$v_3$}}%
\end{picture}%
\caption{$\Rc_7$}
\label{fig:gear}
\end{minipage}
\begin{minipage}{0.45\hsize}
\centering
\unitlength 0.1in
\begin{picture}( 18.7500,  6.8500)(  7.7500,-18.6000)
%
\special{pn 8}%
\special{sh 1.000}%
\special{ar 1000 1400 50 50  0.0000000 6.2831853}%
%
\special{pn 8}%
\special{sh 1.000}%
\special{ar 1000 1800 50 50  0.0000000 6.2831853}%
%
\special{pn 8}%
\special{pa 1000 1400}%
\special{pa 1000 1800}%
\special{fp}%
%
\special{pn 8}%
\special{sh 1.000}%
\special{ar 1400 1400 50 50  0.0000000 6.2831853}%
%
\special{pn 8}%
\special{sh 1.000}%
\special{ar 1400 1800 50 50  0.0000000 6.2831853}%
%
\special{pn 8}%
\special{sh 1.000}%
\special{ar 1800 1800 50 50  0.0000000 6.2831853}%
%
\special{pn 8}%
\special{sh 1.000}%
\special{ar 1800 1400 50 50  0.0000000 6.2831853}%
%
\special{pn 8}%
\special{sh 1.000}%
\special{ar 2200 1400 50 50  0.0000000 6.2831853}%
%
\special{pn 8}%
\special{sh 1.000}%
\special{ar 2200 1800 50 50  0.0000000 6.2831853}%
%
\special{pn 8}%
\special{sh 1.000}%
\special{ar 2600 1600 50 50  0.0000000 6.2831853}%
%
\special{pn 8}%
\special{pa 2600 1600}%
\special{pa 2200 1400}%
\special{fp}%
%
\special{pn 8}%
\special{pa 2200 1400}%
\special{pa 1000 1400}%
\special{fp}%
\special{pa 1000 1400}%
\special{pa 1000 1400}%
\special{fp}%
%
\special{pn 8}%
\special{pa 1000 1800}%
\special{pa 2200 1800}%
\special{fp}%
%
\special{pn 8}%
\special{pa 2200 1800}%
\special{pa 2600 1600}%
\special{fp}%
%
\special{pn 8}%
\special{pa 2200 1800}%
\special{pa 2200 1400}%
\special{fp}%
%
\special{pn 8}%
\special{pa 1800 1400}%
\special{pa 1800 1800}%
\special{fp}%
%
\special{pn 8}%
\special{pa 1400 1800}%
\special{pa 1400 1400}%
\special{fp}%
\put(10.0000,-12.6000){\makebox(0,0){$v_0$}}%
\put(14.0000,-12.6000){\makebox(0,0){$v_1$}}%
\put(18.0000,-12.6000){\makebox(0,0){$v_2$}}%
\put(22.0000,-12.6000){\makebox(0,0){$v_3$}}%
\put(26.0000,-14.6000){\makebox(0,0){$v_4$}}%
\put(22.0500,-19.4500){\makebox(0,0){$v_5$}}%
\put(18.0500,-19.4500){\makebox(0,0){$v_6$}}%
\put(14.0500,-19.4500){\makebox(0,0){$v_7$}}%
\put(10.0500,-19.4500){\makebox(0,0){$v_8$}}%
\end{picture}%
\caption{$\Lc_4$}
\label{fig:pencil}
\end{minipage}
\end{figure}

\begin{prop}\label{prop:odd_exist} \verb||
\begin{enumerate}
\item $c_o(\Rc_k)/c_e(\Rc_k)$ diverges to infinity as $k \to \infty$.
\item $c_e(\Lc_m)/c_o(\Lc_m)$ diverges to infinity as $m \to \infty$.
\end{enumerate}
\end{prop}
\begin{proof}
(1) The number of even cycles in $\Rc_k$ is exactly $k$ (i.e., equal to the number of 4-cycles).
On the other hand, since whether an odd cycle passes an edge shared by $C_k$ and a 4-cycle
does not change the parity of the length of the cycle,
we see that the number of odd cycles is equal to $2^k$. 
Hence, $c_o(\Rc_k)/c_e(\Rc_k) = 2^k/k \to \infty$ as $k \to \infty$.

\medskip
\noindent
(2) Since every odd cycle passes $v_{m-1}v_{m}v_{m+1}$,
the number of odd cycles in $\Lc_m$ is equal to $m$ 
(i.e., equal to the number of added edges in the construction).
On the other hand, since every even cycle passes exactly two added edges,
the number of even cycles is equal to $\binom{m}{2}$. 
Hence, $c_e(\Lc_m)/c_o(\Lc_m) = \binom{m}{2}/m \to \infty$ as $m \to \infty$.
\end{proof}

Next we show a similar result as above for $2$-connected outerplanar graphs $G$ such that $\partial G$ is even.
Let $q \equiv 0 \pmod{4}$ be a positive integer.
Let $\Hc_q$ be the graph obtained from two copies of $\Lc_{q/2}$ 
by identifying two $v_{q/2-1}v_{q/2}$'s of them,
where the labels are as in the construction of $\Lc_{q/2}$; see Figure~\ref{fig:pens}.
Let $\Tc_n$ be the graph obtained from $\Rc_n$ and $\Lc_n$ for some odd $n \geq 3$
by identifying an edge $u_2u_3$ of a 4-cycle $u_0u_1u_2u_3u_0$ in $\Rc_n$ with $u_0u_1 \in E(C_n)$
and an edge $v_0v_{2n}$, where the labels are as in the construction of $\Lc_n$; see Figure~\ref{fig:combin}.
Note that $\Hc_q$ and $\Tc_n$ are $2$-connected non-bipartite outerplanar graphs
and the boundary cycle of each infinite face is even. Moreover, $\Hc_q$ has exactly one odd chord and $\Lc_n$ has exactly $(n+1)$ odd chords. 

\begin{figure}[htb]
\centering
\unitlength 0.1in
\begin{picture}( 29.0000,  7.0000)(  9.5000,-18.5000)
%
\special{pn 8}%
\special{sh 1.000}%
\special{ar 1000 1800 50 50  0.0000000 6.2831853}%
%
\special{pn 8}%
\special{pa 1000 1400}%
\special{pa 1000 1800}%
\special{fp}%
%
\special{pn 8}%
\special{sh 1.000}%
\special{ar 1400 1400 50 50  0.0000000 6.2831853}%
%
\special{pn 8}%
\special{sh 1.000}%
\special{ar 1800 1800 50 50  0.0000000 6.2831853}%
%
\special{pn 8}%
\special{pa 2200 1400}%
\special{pa 1000 1400}%
\special{fp}%
\special{pa 1000 1400}%
\special{pa 1000 1400}%
\special{fp}%
%
\special{pn 8}%
\special{pa 1000 1800}%
\special{pa 2200 1800}%
\special{fp}%
%
\special{pn 8}%
\special{pa 2200 1800}%
\special{pa 2600 1600}%
\special{fp}%
%
\special{pn 8}%
\special{pa 2200 1800}%
\special{pa 2200 1400}%
\special{fp}%
%
\special{pn 8}%
\special{pa 1800 1400}%
\special{pa 1800 1800}%
\special{fp}%
%
\special{pn 8}%
\special{pa 1400 1800}%
\special{pa 1400 1400}%
\special{fp}%
%
\special{pn 8}%
\special{sh 1.000}%
\special{ar 3800 1200 50 50  0.0000000 6.2831853}%
%
\special{pn 8}%
\special{pa 3800 1600}%
\special{pa 3800 1200}%
\special{fp}%
%
\special{pn 8}%
\special{sh 1.000}%
\special{ar 3400 1600 50 50  0.0000000 6.2831853}%
%
\special{pn 8}%
\special{sh 1.000}%
\special{ar 3000 1200 50 50  0.0000000 6.2831853}%
%
\special{pn 8}%
\special{sh 1.000}%
\special{ar 2600 1600 50 50  0.0000000 6.2831853}%
%
\special{pn 8}%
\special{sh 1.000}%
\special{ar 2200 1400 50 50  0.0000000 6.2831853}%
%
\special{pn 20}%
\special{pa 2200 1400}%
\special{pa 2600 1600}%
\special{fp}%
%
\special{pn 8}%
\special{pa 2600 1600}%
\special{pa 3800 1600}%
\special{fp}%
\special{pa 3800 1600}%
\special{pa 3800 1600}%
\special{fp}%
%
\special{pn 8}%
\special{pa 3800 1200}%
\special{pa 2600 1200}%
\special{fp}%
%
\special{pn 8}%
\special{pa 2600 1200}%
\special{pa 2200 1400}%
\special{fp}%
%
\special{pn 8}%
\special{pa 2600 1200}%
\special{pa 2600 1600}%
\special{fp}%
%
\special{pn 8}%
\special{pa 3000 1600}%
\special{pa 3000 1200}%
\special{fp}%
%
\special{pn 8}%
\special{pa 3400 1200}%
\special{pa 3400 1600}%
\special{fp}%
%
\special{pn 8}%
\special{sh 0}%
\special{ar 3800 1600 50 50  0.0000000 6.2831853}%
%
\special{pn 8}%
\special{sh 0}%
\special{ar 3400 1200 50 50  0.0000000 6.2831853}%
%
\special{pn 8}%
\special{sh 0}%
\special{ar 3000 1600 50 50  0.0000000 6.2831853}%
%
\special{pn 8}%
\special{sh 0}%
\special{ar 2600 1200 50 50  0.0000000 6.2831853}%
%
\special{pn 8}%
\special{sh 0}%
\special{ar 2200 1800 50 50  0.0000000 6.2831853}%
%
\special{pn 8}%
\special{sh 0}%
\special{ar 1800 1400 50 50  0.0000000 6.2831853}%
%
\special{pn 8}%
\special{sh 0}%
\special{ar 1400 1800 50 50  0.0000000 6.2831853}%
%
\special{pn 8}%
\special{sh 0}%
\special{ar 1000 1400 50 50  0.0000000 6.2831853}%
\end{picture}%
\caption{$\Hc_8$; a bold line denotes the identified edge.}
\label{fig:pens}
\end{figure}
\begin{figure}[htb]
\centering
\unitlength 0.1in
\begin{picture}( 49.0000, 19.0000)(  9.5000,-24.5000)
%
\special{pn 8}%
\special{sh 1.000}%
\special{ar 2200 1000 50 50  0.0000000 6.2831853}%
%
\special{pn 8}%
\special{sh 1.000}%
\special{ar 1400 1400 50 50  0.0000000 6.2831853}%
%
\special{pn 8}%
\special{sh 1.000}%
\special{ar 1800 600 50 50  0.0000000 6.2831853}%
%
\special{pn 8}%
\special{sh 1.000}%
\special{ar 2000 2200 50 50  0.0000000 6.2831853}%
%
\special{pn 8}%
\special{sh 1.000}%
\special{ar 1400 800 50 50  0.0000000 6.2831853}%
%
\special{pn 8}%
\special{sh 1.000}%
\special{ar 1000 1800 50 50  0.0000000 6.2831853}%
%
\special{pn 8}%
\special{sh 1.000}%
\special{ar 3000 1800 50 50  0.0000000 6.2831853}%
%
\special{pn 8}%
\special{sh 1.000}%
\special{ar 3000 1400 50 50  0.0000000 6.2831853}%
%
\special{pn 8}%
\special{sh 1.000}%
\special{ar 2800 1000 50 50  0.0000000 6.2831853}%
%
\special{pn 8}%
\special{sh 1.000}%
\special{ar 2600 2200 50 50  0.0000000 6.2831853}%
%
\special{pn 8}%
\special{sh 1.000}%
\special{ar 1400 2200 50 50  0.0000000 6.2831853}%
%
\special{pn 8}%
\special{pa 2600 1800}%
\special{pa 2000 2200}%
\special{fp}%
%
\special{pn 8}%
\special{pa 2000 2200}%
\special{pa 1400 1800}%
\special{fp}%
%
\special{pn 8}%
\special{pa 1400 1800}%
\special{pa 1400 1400}%
\special{fp}%
%
\special{pn 8}%
\special{pa 1400 1400}%
\special{pa 1800 1000}%
\special{fp}%
%
\special{pn 8}%
\special{pa 1800 1000}%
\special{pa 2200 1000}%
\special{fp}%
%
\special{pn 8}%
\special{pa 2200 1000}%
\special{pa 2600 1400}%
\special{fp}%
%
\special{pn 8}%
\special{pa 2600 1400}%
\special{pa 2600 1800}%
\special{fp}%
%
\special{pn 8}%
\special{pa 2600 1400}%
\special{pa 2800 1000}%
\special{fp}%
%
\special{pn 8}%
\special{pa 2800 1000}%
\special{pa 2600 800}%
\special{fp}%
%
\special{pn 8}%
\special{pa 2600 800}%
\special{pa 2200 1000}%
\special{fp}%
%
\special{pn 8}%
\special{pa 2200 1000}%
\special{pa 2200 600}%
\special{fp}%
%
\special{pn 8}%
\special{pa 2200 600}%
\special{pa 1800 600}%
\special{fp}%
%
\special{pn 8}%
\special{pa 1800 600}%
\special{pa 1800 1000}%
\special{fp}%
%
\special{pn 8}%
\special{pa 1800 1000}%
\special{pa 1400 800}%
\special{fp}%
%
\special{pn 8}%
\special{pa 1400 800}%
\special{pa 1200 1000}%
\special{fp}%
%
\special{pn 8}%
\special{pa 1200 1000}%
\special{pa 1400 1400}%
\special{fp}%
%
\special{pn 8}%
\special{pa 1400 1400}%
\special{pa 1000 1400}%
\special{fp}%
%
\special{pn 8}%
\special{pa 1000 1400}%
\special{pa 1000 1800}%
\special{fp}%
%
\special{pn 8}%
\special{pa 1000 1800}%
\special{pa 1400 1800}%
\special{fp}%
%
\special{pn 8}%
\special{pa 1400 1800}%
\special{pa 1400 2200}%
\special{fp}%
%
\special{pn 8}%
\special{pa 1400 2200}%
\special{pa 1600 2400}%
\special{fp}%
%
\special{pn 8}%
\special{pa 1600 2400}%
\special{pa 2000 2200}%
\special{fp}%
%
\special{pn 8}%
\special{pa 2000 2200}%
\special{pa 2400 2400}%
\special{fp}%
%
\special{pn 8}%
\special{pa 2400 2400}%
\special{pa 2600 2200}%
\special{fp}%
%
\special{pn 8}%
\special{pa 2600 2200}%
\special{pa 2600 1800}%
\special{fp}%
%
\special{pn 8}%
\special{pa 2600 1800}%
\special{pa 3000 1800}%
\special{fp}%
%
\special{pn 20}%
\special{pa 3000 1800}%
\special{pa 3000 1400}%
\special{fp}%
%
\special{pn 8}%
\special{pa 3000 1400}%
\special{pa 2600 1400}%
\special{fp}%
%
\special{pn 8}%
\special{pa 3400 1400}%
\special{pa 3400 1800}%
\special{fp}%
%
\special{pn 8}%
\special{pa 3800 1800}%
\special{pa 3800 1400}%
\special{fp}%
%
\special{pn 8}%
\special{pa 4200 1400}%
\special{pa 4200 1800}%
\special{fp}%
%
\special{pn 8}%
\special{pa 4600 1800}%
\special{pa 4600 1400}%
\special{fp}%
%
\special{pn 8}%
\special{pa 5000 1400}%
\special{pa 5000 1800}%
\special{fp}%
%
\special{pn 8}%
\special{pa 5400 1800}%
\special{pa 5400 1400}%
\special{fp}%
%
\special{pn 8}%
\special{pa 3000 1400}%
\special{pa 5400 1400}%
\special{fp}%
%
\special{pn 8}%
\special{pa 5400 1400}%
\special{pa 5800 1600}%
\special{fp}%
%
\special{pn 8}%
\special{pa 5800 1600}%
\special{pa 5400 1800}%
\special{fp}%
%
\special{pn 8}%
\special{pa 5400 1800}%
\special{pa 3000 1800}%
\special{fp}%
%
\special{pn 8}%
\special{sh 1.000}%
\special{ar 3800 1800 50 50  0.0000000 6.2831853}%
%
\special{pn 8}%
\special{sh 1.000}%
\special{ar 3800 1400 50 50  0.0000000 6.2831853}%
%
\special{pn 8}%
\special{sh 1.000}%
\special{ar 4600 1800 50 50  0.0000000 6.2831853}%
%
\special{pn 8}%
\special{sh 1.000}%
\special{ar 4600 1400 50 50  0.0000000 6.2831853}%
%
\special{pn 8}%
\special{sh 1.000}%
\special{ar 5400 1800 50 50  0.0000000 6.2831853}%
%
\special{pn 8}%
\special{sh 1.000}%
\special{ar 5400 1400 50 50  0.0000000 6.2831853}%
%
\special{pn 8}%
\special{sh 0}%
\special{ar 2600 1800 50 50  0.0000000 6.2831853}%
%
\special{pn 8}%
\special{sh 0}%
\special{ar 2600 1400 50 50  0.0000000 6.2831853}%
%
\special{pn 8}%
\special{sh 0}%
\special{ar 3400 1400 50 50  0.0000000 6.2831853}%
%
\special{pn 8}%
\special{sh 0}%
\special{ar 3400 1800 50 50  0.0000000 6.2831853}%
%
\special{pn 8}%
\special{sh 0}%
\special{ar 4200 1800 50 50  0.0000000 6.2831853}%
%
\special{pn 8}%
\special{sh 0}%
\special{ar 4200 1400 50 50  0.0000000 6.2831853}%
%
\special{pn 8}%
\special{sh 0}%
\special{ar 5000 1400 50 50  0.0000000 6.2831853}%
%
\special{pn 8}%
\special{sh 0}%
\special{ar 5000 1800 50 50  0.0000000 6.2831853}%
%
\special{pn 8}%
\special{sh 0}%
\special{ar 5800 1600 50 50  0.0000000 6.2831853}%
%
\special{pn 8}%
\special{sh 0}%
\special{ar 2600 800 50 50  0.0000000 6.2831853}%
%
\special{pn 8}%
\special{sh 0}%
\special{ar 2200 600 50 50  0.0000000 6.2831853}%
%
\special{pn 8}%
\special{sh 0}%
\special{ar 1800 1000 50 50  0.0000000 6.2831853}%
%
\special{pn 8}%
\special{sh 0}%
\special{ar 1200 1000 50 50  0.0000000 6.2831853}%
%
\special{pn 8}%
\special{sh 0}%
\special{ar 1000 1400 50 50  0.0000000 6.2831853}%
%
\special{pn 8}%
\special{sh 0}%
\special{ar 1400 1800 50 50  0.0000000 6.2831853}%
%
\special{pn 8}%
\special{sh 0}%
\special{ar 1600 2400 50 50  0.0000000 6.2831853}%
%
\special{pn 8}%
\special{sh 0}%
\special{ar 2400 2400 50 50  0.0000000 6.2831853}%
\end{picture}%
\caption{$\Tc_7$; a bold line denotes the identified edge.}
\label{fig:combin}
\end{figure}

To show the desired claim, we prepare the following useful observation.
For a graph $G$ and an edge $f \in E(G)$, 
let $c_o(G,f)$ and $c_e(G,f)$ denote the number of odd cycles and 
that of even cycles passing $f$ in $G$, respectively.

\begin{obs}\label{obs:addedge}
Let $G_1$ and $G_2$ be graphs, and let $e_1$ and $e_2$ be edges of $G_1$ and $G_2$, respectively.
Let $G$ be the graph obtained from $G_1$ and $G_2$ by identifying $e_1$ and $e_2$. 
Then the following equalities hold:
\begin{align}
c_o(G) &= c_o(G_1) + c_o(G_2) + c_o(G_1,e_1)c_e(G_2,e_2) + c_e(G_1,e_1)c_o(G_2,e_2) \label{eq01}\\
c_e(G) &= c_e(G_1) + c_e(G_2) + c_o(G_1,e_1)c_o(G_2,e_2) + c_e(G_1,e_1)c_e(G_2,e_2) \label{eq02}
\end{align}
\end{obs}

\begin{prop}\label{prop:even_exist} \verb||
\begin{enumerate}
\item $c_o(\Tc_n)/c_e(\Tc_n)$ diverges to infinity as $n \to \infty$.
\item $c_e(\Hc_q)/c_o(\Hc_q)$ diverges to infinity as $q \to \infty$.
\end{enumerate}
\end{prop}
\begin{proof}
By the proof of Proposition~\ref{prop:odd_exist}, we have 
\begin{align*}
c_o(\Rc_n) = 2^{n}, \quad c_e(\Rc_n) = n, \quad 
c_o(\Lc_n) = n \quad \mbox{and} \quad c_e(\Lc_n) = \binom{n}{2}.
\end{align*}

Let $e$ be an edge of a 4-cycle in $\Rc_n$ with each end not lying on $C_n$.
Let $f_1 = v_0v_{2n}$ and $f_2 = v_{n-1}v_n$ be the edges of $\Lc_n$ 
where the labels of vertices are as in its construction.
Then we have
\begin{align*}
c_o(\Rc_n,e) = 2^{n-1} \quad &\mbox{and} \quad c_e(\Rc_n,e) = 1,\\
c_o(\Lc_n,f_1) = 1 \quad &\mbox{and} \quad c_e(\Lc_n,f_1) = n-1,\\
c_o(\Lc_n,f_2) = n \quad &\mbox{and} \quad c_e(\Lc_n,f_2) = 0.
\end{align*}
Thus, by Observation~\ref{obs:addedge}, we have
\begin{align*}
\frac{c_o(\Tc_n)}{c_e(\Tc_n)}
&= \frac{c_o(\Rc_n) + c_o(\Lc_n) + c_o(\Rc_n,e)c_e(\Lc_n,f_1) + c_e(\Rc_n,e)c_o(\Lc_n,f_1)}
{c_e(\Rc_n) + c_e(\Lc_n) + c_o(\Rc_n,e)c_o(\Lc_n,f_1) + c_e(\Rc_n,e)c_e(\Lc_n,f_1)}\\
&=\frac{2^n + n + 2^{n-1}(n-1) + 1}{n +  \binom{n}{2} + 2^{n-1} + n-1}\\
&=\frac{n + 1 + (n+1)/2^{n-1}}{1 + (n^2+3n-2)/2^n} \to \infty \quad \mbox{as $n \to \infty$},
\end{align*}
and 
\begin{align*}
\frac{c_e(\Hc_q)}{c_o(\Hc_q)}
&= \frac
{2c_e(\Lc_q) + c_o(\Lc_q,f_2)^2 + c_e(\Lc_q,f_2)^2}
{2c_o(\Lc_q) + 2c_o(\Lc_q,f_2)c_e(\Lc_q,f_2)}
= q - \frac{1}{2} \to \infty \quad \mbox{as $q \to \infty$}.
\end{align*}
Therefore, the proposition holds.
\end{proof}

The following theorem is a direct consequence of Propositions~\ref{prop:odd_exist} and~\ref{prop:even_exist}.

\begin{thm}\label{thm:exist}
There are $2$-connected outerplanar graphs $G$ such that 
for any $r \in \{c_o(G)/c_e(G), c_e(G)/c_o(G)\}$,
$r$ diverges to infinity as $|V(G)| \to \infty$, not depending on the parity of $\partial G$.
\end{thm}

\section{Bounds}\label{sec:bounds}

We first consider outerplanar graphs all of whose finite faces are odd, 
where a finite face $f$ is {\it odd} (resp.~{\it even}) 
if the length of the boundary walk of $f$ is odd (resp.~even).
In this case, we can directly apply Theorem~\ref{thm:subtrees} 
to evaluate the number of odd/even cycles.

\begin{thm}\label{thm:oddfaces}
Let $G$ be a $2$-connected outerplanar graph with each finite face odd and $|F(G)| \geq 3$.
Then $\frac{|F(G)|-1}{2c_e(G)} + 1 \leq c_o(G)/c_e(G) \leq \frac{|F(G)|-2}{c_e(G)} + 1$. 
In particular, $1 \leq c_o(G)/c_e(G) \leq 2$. 
\end{thm}
\begin{proof}
By the assumption, for any subtree $H$ of the dual tree $T_G$,
the parity of the order of $H$ is the same as that of 
the cycle in $G$ bounding a union of faces corresponding to vertices of $H$.
Thus, Theorem~\ref{thm:subtrees} implies that $c_o(G) - c_e(G) = \alpha(T_G)$,
where $\alpha(T_G)$ denotes the independence number of $T_G$.
Since $\frac{|F(G)|-1}{2} \leq \alpha(T_G) \leq (|F(G)|-1) - 1$,
we obtain that $1\leq\frac{|F(G)|-1}{2c_e(G)} + 1 \leq c_o(G)/c_e(G) \leq \frac{|F(G)|-2}{c_e(G)} + 1$.
Since $c_e(G)$ is equal to the number of even subtrees of $T_G$, 
one has $c_e(G) \geq |E(T_G)| = |F(G)|-2 \geq 1$ by $|F(G)| \geq 3$. 
Hence, $c_o(G)/c_e(G) \leq 2$ also holds. 
\end{proof}



Next, we give a general upper bound of $c_o(G)/c_e(G)$ for a $2$-connected outerplanar graph $G$ with even $\partial G$ as follows.

\begin{thm}\label{thm:oddchords}
Let $G$ be a $2$-connected outerplanar graph with even $\partial G$ and let $k \geq 0$ be the number of odd chords of $G$. 
Suppose that $G$ has at least one even chord. Then $c_o(G)/c_e(G) \leq k$.
Furthermore, this bound is sharp.
\end{thm}
\begin{proof}
We prove the theorem by induction on $k$.
If $k = 0$, then $c_o(G)/c_e(G) = 0$, so we are done.
Assume that $k\geq 1$ and the theorem holds for any non-negative integer smaller than $k$.

Let $e$ be an odd chord of $G$ 
and let $G'$ be the 2-connected outerplanar graph obtained from $G$ by removing $e$.
Since $G'$ has $k-1$ odd chords and at least one even chord, $c_o(G')/c_e(G') \leq k - 1$ by induction hypothesis.
On the other hand, $G$ can be obtained from two outerplanar graphs $G_1$ and $G_2$ by identifying $e$,
where $e$ lies on both $\partial G_1$ and $\partial G_2$. 
Thus, $c_e(G') = c_o(G_1,e)c_o(G_2,e) + \alpha$, where $\alpha \geq 1$ since $G'$ has at least one even chord. 
By $c_o(G')/c_e(G') \leq k - 1$, we know $c_o(G') \leq (k-1) \left( c_o(G_1,e)c_o(G_2,e) + \alpha \right)$.
Note that $c_e(G) \geq c_e(G')$. Therefore, we have
\begin{align*}
\frac{c_o(G)}{c_e(G)} &\leq 
\frac{(k-1) \left( c_o(G_1,e)c_o(G_2,e) + \alpha\right) + c_o(G_1,e) + c_o(G_2,e)}
{c_o(G_1,e)c_o(G_2,e) + \alpha}\\
&\leq (k-1) + \frac{c_o(G_1,e) + c_o(G_2,e)}{c_o(G_1,e)c_o(G_2,e) + \alpha} \leq k,
\end{align*}
where the final inequality holds by $\alpha \geq 1$ and $c_o(G_1,e), c_o(G_2,e) \geq 1$. 

Regarding the sharpness of $c_o(G)/c_e(G) \leq k$, 
since $\Tc_k$ has $(k+1)$ odd chords and at least one even chord when $k \geq 3$, 
we see that $c_o(\Tc_k)/c_e(\Tc_k)=\frac{k + 1 + (k+1)/2^{k-1}}{1 + (k^2+3k-2)/2^k} \geq k$ if $k \geq 11$. 
Therefore, this upper bound is the best possible.
\end{proof}

In the remaining part of this section, we show upper/lower bounds of the ratio
for outerplanar graphs with prescribed dual trees,
namely a path and a star.
\begin{prop}\label{prop:dualpath}
Let $G$ be a $2$-connected outerplanar graph with even $\partial G$ and assume that the dual tree $T_G$ of $G$ is a path. 
Then $c_o(G)/c_e(G) \leq 2$. Furthermore, this bound is sharp.
\end{prop}
\begin{proof}
Let $a$ and $b$ be the number of odd and even chords of $G$, respectively.
Since $T_G$ is a path, we have
\begin{align*}
c_o(G) = a \cdot b +2a = a (b+2) \; \mbox{ and }\; c_e(G) = \binom{b}{2} + \binom{a}{2} + 2b + 1.
\end{align*}
We suppose to the contrary that $c_o(G) / c_e(G) > 2$. 
This together with the above two equations leads to the following: 
\begin{align*}
\frac{c_o(G)}{c_e(G)}=\frac{2a(b+2)}{a(a-1)+(b+1)(b+2)} > 2 \;\Longleftrightarrow\; a(b+2)>a(a-1)+(b+1)(b+2). 
\end{align*}
By this inequality, we see that $a(b-a+3)>(b+1)(b+2) > 0$, so $b-a+3 \geq 1$, i.e., $a \leq b+2$. 
We also see that $(a-b-1)(b+2)>a(a-1) \geq 0$, so $a \geq b+1$. Thus, we obtain that $a=b+1$ or $b+2$. 
However, for each of those two cases, we have $c_o(G)/c_e(G)=(b+2)/(b+1) \leq 2$, a contradiction. 

Regarding the sharpness of this bound, 
every outerplanar graph $G$ obtained from an even cycle by adding exactly one odd chord attains the equality. 
\end{proof}

\begin{prop}\label{prop:dualstar}
Let $G$ be a $2$-connected outerplanar graph and assume that the dual tree $T_G$ of $G$ is a star.
If there is at least one odd face corresponding to a leaf of $T_G$,
then both $c_o(G)/c_e(G)$ and $c_e(G)/c_o(G)$ converge to $1$ as $|F(G)| \to \infty$. 
\end{prop}
\begin{proof}
First suppose that $\partial G$ is even.
Let $a$ and $b$ be the number of odd chords and that of even chords of $G$, respectively;
note that $a \geq 1$ by the assumption.
Let $f$ be the face in $G$ corresponding to the center vertex of $T_G$. 
Since $T_G$ is a star and $\partial G$ is even, 
the parity of the number of odd chords lying on $\partial f$ 
is the same as that of the length of $\partial f$. 
If $a$ is even, then we have 
\begin{align*}
c_o(G) &= 2^b \sum^{a/2 - 1}_{k=0} \binom{a}{2k+1} + a = 2^{a+b-1}+a, \mbox{ and } \\
c_e(G) &= 2^b \sum^{a/2}_{k=0} \binom{a}{2k} + b = 2^{a+b-1}+b.
\end{align*}
If $a$ is odd, then we have 
\begin{align*}
c_o(G) &= 2^b \sum^{(a-1)/2}_{k=0} \binom{a}{2k} + a = 2^{a+b-1}+a, \mbox{ and } \\
c_e(G) &= 2^b \sum^{(a-1)/2}_{k=0} \binom{a}{2k+1} + b = 2^{a+b-1}+b.
\end{align*}
Thus, regardless of the parity of $a$,
$c_o(G)/c_e(G)$ and $c_e(G)/c_o(G)$ converge to $1$ as $|F(G)| \to \infty$ (i.e., $a+b \to \infty$)
since $2^{a+b-1} \gg a+b$.

Next suppose that $\partial G$ is odd.
Let $f$ be the face in $G$ corresponding to the center vertex of $T_G$.
Let $a$ (resp.~$b$) be the number of chords shared by $f$ 
and an odd (resp.~even) face corresponding to a leaf of $T_G$.
Since $T_G$ is a star and $\partial G$ is odd, 
the parity of the length of $\partial f$ is opposite to the parity of the number of odd chords lying on $\partial f$.
Thus, the number of odd and even cycles can be calculated similarly to the first case
depending on the opposite parity of $a$.
\end{proof}

\section{Forbidden subgraphs/minors condition}\label{sec:forbidden}

A graph $G$ is {\it $H$-free} (resp.~{\it $H$-minor-free})
if $G$ contains no $H$ as its induced subgraph (resp.~as a minor).
Recall that Theorem~\ref{thm:exist} implying that 
the $K_4$-minor-free and $K_{2,3}$-minor-free condition is not sufficient 
to bound $c_o(G)/c_e(G)$ or $c_e(G)/c_o(G)$ by a constant for a 2-connected outerplanar graph $G$.
Thus, we consider another forbidden subgraphs/minors condition 
which bounds $c_o(G)/c_e(G)$ or $c_e(G)/c_o(G)$ by a constant.

It is easy to see that for any $k \geq 3$,
there is a 2-connected $C_k$-free outerplanar graph such that 
neither $c_o(G)/c_e(G)$ nor $c_e(G)/c_o(G)$ can be bounded by a constant
since we can make each cycle in the constructions of $\Rc$ and $\Lc$ be arbitrarily long.
Of course, no 2-connected graph of order at least~3 is $C_3$-minor-free,
and every $H$-minor-free graph is $H$-free.
Therefore, we focus on a 2-connected $K_{1,3}$-free (or claw-free) and $K_4$-minor-free graphs.

We here introduce two particular outerplanar graphs.
Let $C_t = v_0v_1 \dots v_{t-1}$ be a cycle with $t \geq 3$.
For several indices $i \in \{0,1,\dots,t-1\}$,
we add a vertex $r_i$ to make a triangle $v_iv_{i+1}r_i$, where subscripts are modulo $t$
(see Figure~\ref{fig:sungear}).
The set of $2^t$ resulting graphs constructed above is denoted by $\Sc_t$.
Let $\Zc_d$ be the outerplanar graph obtained from $\Lc_d$ (where recall that $d \geq 1$)
by joining $v_i$ and $v_{2d-1-i}$ for each $i \in \{0,\dots,d-2\}$ (see Figure~\ref{fig:pencilZ}),
and let $\Zc^*_d$ be the graph obtained from $\Zc_d$ by removing $v_d$.
Observe that $\Zc_d, \Zc^*_d$ 
and any graph in $\Sc_t$ are 2-connected outerplanar graphs and that they are $K_{1,3}$-free.

\begin{figure}[htb]
\begin{minipage}{0.5\hsize}
\centering
\unitlength 0.1in
\begin{picture}( 18.9500, 15.3500)(  9.5500,-22.5000)
%
\special{pn 8}%
\special{sh 1.000}%
\special{ar 1800 1000 50 50  0.0000000 6.2831853}%
%
\special{pn 8}%
\special{sh 1.000}%
\special{ar 2200 1000 50 50  0.0000000 6.2831853}%
%
\special{pn 8}%
\special{sh 1.000}%
\special{ar 1400 1400 50 50  0.0000000 6.2831853}%
%
\special{pn 8}%
\special{sh 1.000}%
\special{ar 1400 1800 50 50  0.0000000 6.2831853}%
%
\special{pn 8}%
\special{sh 1.000}%
\special{ar 2000 2200 50 50  0.0000000 6.2831853}%
%
\special{pn 8}%
\special{sh 1.000}%
\special{ar 2600 1800 50 50  0.0000000 6.2831853}%
%
\special{pn 8}%
\special{sh 1.000}%
\special{ar 2600 1400 50 50  0.0000000 6.2831853}%
%
\special{pn 8}%
\special{sh 1.000}%
\special{ar 2800 1600 50 50  0.0000000 6.2831853}%
%
\special{pn 8}%
\special{pa 2600 1800}%
\special{pa 2000 2200}%
\special{fp}%
%
\special{pn 8}%
\special{pa 2000 2200}%
\special{pa 1400 1800}%
\special{fp}%
%
\special{pn 8}%
\special{pa 1400 1800}%
\special{pa 1400 1400}%
\special{fp}%
%
\special{pn 8}%
\special{pa 1400 1400}%
\special{pa 1800 1000}%
\special{fp}%
%
\special{pn 8}%
\special{pa 1800 1000}%
\special{pa 2200 1000}%
\special{fp}%
%
\special{pn 8}%
\special{pa 2200 1000}%
\special{pa 2600 1400}%
\special{fp}%
%
\special{pn 8}%
\special{pa 2600 1400}%
\special{pa 2600 1800}%
\special{fp}%
\put(20.0000,-20.6000){\makebox(0,0){$v_0$}}%
\put(15.8000,-18.0000){\makebox(0,0){$v_1$}}%
\put(15.8000,-14.0000){\makebox(0,0){$v_2$}}%
\put(24.3000,-14.0000){\makebox(0,0){$v_5$}}%
\put(24.3000,-18.0000){\makebox(0,0){$v_6$}}%
\put(22.0000,-11.4000){\makebox(0,0){$v_4$}}%
\put(18.0000,-11.4000){\makebox(0,0){$v_3$}}%
%
\special{pn 8}%
\special{sh 1.000}%
\special{ar 1600 2200 50 50  0.0000000 6.2831853}%
%
\special{pn 8}%
\special{sh 1.000}%
\special{ar 1200 1600 50 50  0.0000000 6.2831853}%
%
\special{pn 8}%
\special{sh 1.000}%
\special{ar 2000 800 50 50  0.0000000 6.2831853}%
%
\special{pn 8}%
\special{pa 2800 1600}%
\special{pa 2600 1400}%
\special{fp}%
%
\special{pn 8}%
\special{pa 2800 1600}%
\special{pa 2600 1800}%
\special{fp}%
%
\special{pn 8}%
\special{pa 2000 2200}%
\special{pa 1600 2200}%
\special{fp}%
%
\special{pn 8}%
\special{pa 1600 2200}%
\special{pa 1400 1800}%
\special{fp}%
%
\special{pn 8}%
\special{pa 1400 1800}%
\special{pa 1200 1600}%
\special{fp}%
%
\special{pn 8}%
\special{pa 1200 1600}%
\special{pa 1400 1400}%
\special{fp}%
%
\special{pn 8}%
\special{pa 1800 1000}%
\special{pa 2000 800}%
\special{fp}%
%
\special{pn 8}%
\special{pa 2000 800}%
\special{pa 2200 1000}%
\special{fp}%
\put(14.3000,-22.0000){\makebox(0,0){$r_0$}}%
\put(11.8000,-14.7000){\makebox(0,0){$r_1$}}%
\put(21.7000,-8.0000){\makebox(0,0){$r_3$}}%
\put(28.1000,-14.7000){\makebox(0,0){$r_5$}}%
\end{picture}%
\caption{A graph in $\Sc_7$}
\label{fig:sungear}
\end{minipage}
\begin{minipage}{0.45\hsize}
\centering
\unitlength 0.1in
\begin{picture}( 18.7500,  6.8500)(  7.7500,-18.6000)
%
\special{pn 8}%
\special{sh 1.000}%
\special{ar 1000 1400 50 50  0.0000000 6.2831853}%
%
\special{pn 8}%
\special{sh 1.000}%
\special{ar 1000 1800 50 50  0.0000000 6.2831853}%
%
\special{pn 8}%
\special{pa 1000 1400}%
\special{pa 1000 1800}%
\special{fp}%
%
\special{pn 8}%
\special{sh 1.000}%
\special{ar 1400 1400 50 50  0.0000000 6.2831853}%
%
\special{pn 8}%
\special{sh 1.000}%
\special{ar 1400 1800 50 50  0.0000000 6.2831853}%
%
\special{pn 8}%
\special{sh 1.000}%
\special{ar 1800 1800 50 50  0.0000000 6.2831853}%
%
\special{pn 8}%
\special{sh 1.000}%
\special{ar 1800 1400 50 50  0.0000000 6.2831853}%
%
\special{pn 8}%
\special{sh 1.000}%
\special{ar 2200 1400 50 50  0.0000000 6.2831853}%
%
\special{pn 8}%
\special{sh 1.000}%
\special{ar 2200 1800 50 50  0.0000000 6.2831853}%
%
\special{pn 8}%
\special{sh 1.000}%
\special{ar 2600 1600 50 50  0.0000000 6.2831853}%
%
\special{pn 8}%
\special{pa 2600 1600}%
\special{pa 2200 1400}%
\special{fp}%
%
\special{pn 8}%
\special{pa 2200 1400}%
\special{pa 1000 1400}%
\special{fp}%
\special{pa 1000 1400}%
\special{pa 1000 1400}%
\special{fp}%
%
\special{pn 8}%
\special{pa 1000 1800}%
\special{pa 2200 1800}%
\special{fp}%
%
\special{pn 8}%
\special{pa 2200 1800}%
\special{pa 2600 1600}%
\special{fp}%
%
\special{pn 8}%
\special{pa 2200 1800}%
\special{pa 2200 1400}%
\special{fp}%
%
\special{pn 8}%
\special{pa 1800 1400}%
\special{pa 1800 1800}%
\special{fp}%
%
\special{pn 8}%
\special{pa 1400 1800}%
\special{pa 1400 1400}%
\special{fp}%
\put(10.0000,-12.6000){\makebox(0,0){$v_0$}}%
\put(14.0000,-12.6000){\makebox(0,0){$v_1$}}%
\put(18.0000,-12.6000){\makebox(0,0){$v_2$}}%
\put(22.0000,-12.6000){\makebox(0,0){$v_3$}}%
\put(26.0000,-14.6000){\makebox(0,0){$v_4$}}%
\put(22.0500,-19.4500){\makebox(0,0){$v_5$}}%
\put(18.0500,-19.4500){\makebox(0,0){$v_6$}}%
\put(14.0500,-19.4500){\makebox(0,0){$v_7$}}%
\put(10.0500,-19.4500){\makebox(0,0){$v_8$}}%
%
\special{pn 8}%
\special{pa 1000 1400}%
\special{pa 1400 1800}%
\special{fp}%
%
\special{pn 8}%
\special{pa 1400 1400}%
\special{pa 1800 1800}%
\special{fp}%
%
\special{pn 8}%
\special{pa 1800 1400}%
\special{pa 2200 1800}%
\special{fp}%
\end{picture}%
\caption{$\Zc_4$}
\label{fig:pencilZ}
\end{minipage}
\end{figure}

\begin{thm}\label{thm:forbid_claw_k4}
If a $2$-connected graph $G$ is $K_{1,3}$-free and $K_4$-minor-free, then $G$ is outerplanar. 
Furthermore, $G$ is isomorphic to $\Zc_d, \Zc^*_d$ or a graph in $\Sc_t$.
\end{thm}
\begin{proof}
Let $G$ be a 2-connected $K_{1,3}$-free and $K_4$-minor-free graph.
Note that $G$ has no vertex of degree at least~5;
otherwise, $G$ has an induced $K_{1,3}$ or a $K_4$-minor consisting of the vertex and its neighbors.
Thus, $\deg_G(x) \leq 4$ for every vertex $x$ of $G$,
where $\deg_G(x)$ denotes the degree of a vertex $x$ in $G$.

It is well known that every $K_4$-minor-free graph is a subgraph of a 2-tree~\cite{Dirac},
where a {\it $2$-tree} is a graph obtained from a triangle by repeatedly adding vertices
in such a way that each added vertex has two adjacent neighbors (i.e., those three vertices induce a triangle).
Hence, since every 2-tree is planar,
we embed $G$ into the plane in such a way that the length of the boundary cycle of the outer face is the longest 
among all planar embeddings of $G$.

Let $C=u_0u_1 \dots u_{m-1}$ for some $m \geq 3$ be the outer cycle of $G$.
If all vertices of $G$ lie on $C$, then $G$ is outerplanar.
Thus, we may suppose that $u_i \in V(C)$ has degree~3 or~4
and to the contrary that $u_i$ has a neighbor not lying on $C$.
Let $K_4(a,b,c,d)$ denote a $K_4$-minor in $G$ 
consisting of four vertices $a,b,c,d$ and internally disjoint paths between them,
where two paths are {\it internally disjoint} if they do not share vertices except their end vertices.

First suppose $\deg_G(u_i) = 3$.
Let $a$ be a unique neighbor of $u_i$ which does not lie on $C$.
A path $P$ between two vertices $x$ and $y$ 
in which any vertex of $P$ does not lie on $C$ except $x,y$ 
is called an {\it inner $(x,y)$-path}.
Since $G$ is $K_{1,3}$-free, 
we may assume by symmetry that $u_{i-1}a \in E(G)$ or $u_{i-1}u_{i+1} \in E(G)$.
Moreover, since $G$ is $K_4$-minor-free, 
there is at most one of an inner $(a,u_{i-1})$-path and an inner $(a,u_{i+1})$-path, say the former.

\begin{figure}[htb]
\centering
\input{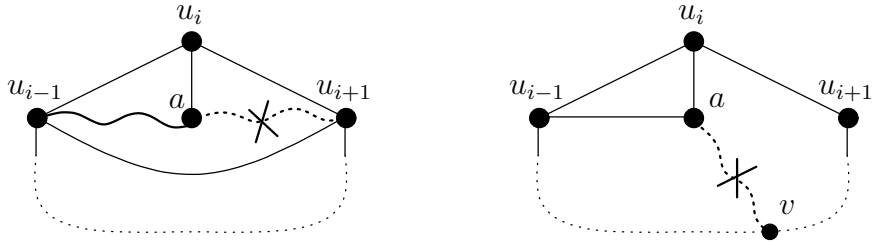}
\caption{The case when $\deg_G(u_i) = 3$}
\label{t41-deg3}
\end{figure}

See Figure~\ref{t41-deg3}.
The left depicts the case when $u_{i-1}u_{i+1} \in E(G)$ and there is an inner $(a,u_{i-1})$-path.
In this case, if there is also an inner $(a,u_{i+1})$-path, 
then $G$ has a $K_4$-minor $K_4(u_{i-1},u_i,u_{i+1},a)$, a contradiction.
The right depicts the case when $u_{i-1}a \in E(G)$ but $u_{i-1}u_{i+1} \notin E(G)$.
In this case, if there is an inner $(a,v)$-path for some $v \in V(C) \setminus \{u_{i-1},u_{i}\}$, 
then $G$ has a $K_4$-minor $K_4(u_{i-1},u_i,v,a)$.
Thus, in either case,
we can obtain another outerplanar embedding of $G$ with outer cycle longer than $C$,
by making the outer cycle pass $u_ia$ and an inner $(a,u_{i-1})$-path instead of $u_{i-1}u_{i}$.
Intuitively, we can obtain this embedding 
by applying a ``jump" of the edge $u_{i-1}u_i$ over $a$ as shown in Figure~\ref{t41-jump}.
This contradicts that $C$ is the longest outer cycle.

\begin{figure}[htb]
\centering
\input{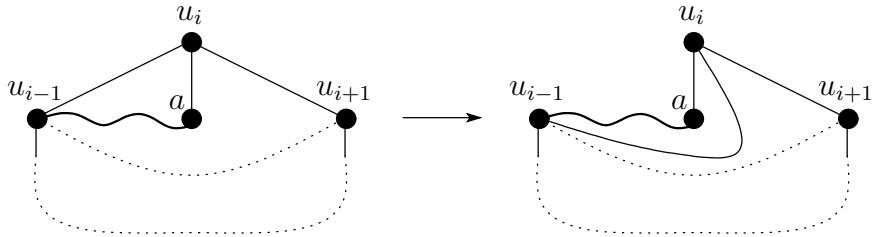}
\caption{A ``jump" of $u_{i-1}u_i$}
\label{t41-jump}
\end{figure}

Next suppose $\deg_G(u_i) = 4$.
Let $a,b$ be neighbors of $u_i$ with $a,b \notin \{u_{i-1},u_i\}$.
By symmetry, we first consider the case when $b \in V(C)$ (see Figure~\ref{t41-deg4-b}).
Since four vertices $u_{i-1},u_i,u_{i+1},a$ do not induce a $K_{1,3}$, $u_{i-1}a \in E(G)$.
Thus, $G$ has no inner $(a,v)$-path for any $v \in V(C) \setminus \{u_{i-1},u_{i}\}$;
otherwise $G$ has a $K_4$-minor $K_4(u_{i-1},u_i,a,v)$.
Therefore, similarly to the case when $\deg_G(u_i) = 3$,
we can obtain another outerplanar embedding of $G$ with outer cycle longer than $C$
by applying a ``jump" of the edge $u_{i-1}u_i$ over $a$.

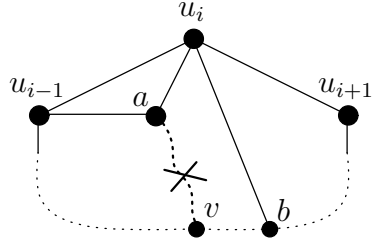
\begin{figure}[htb]
\centering
\unitlength 0.1in
\begin{picture}( 20.6500, 12.6000)( 15.9000,-20.4000)
%
\special{pn 13}%
\special{pa 2690 1800}%
\special{pa 2820 1650}%
\special{fp}%
%
\special{pn 13}%
\special{pa 2654 1690}%
\special{pa 2870 1744}%
\special{fp}%
\put(25.3000,-13.2000){\makebox(0,0){$a$}}%
%
\special{pn 8}%
\special{sh 1.000}%
\special{ar 2006 1406 50 50  0.0000000 6.2831853}%
%
\special{pn 8}%
\special{sh 1.000}%
\special{ar 2806 1006 50 50  0.0000000 6.2831853}%
%
\special{pn 8}%
\special{sh 1.000}%
\special{ar 3606 1406 50 50  0.0000000 6.2831853}%
%
\special{pn 8}%
\special{pa 3606 1406}%
\special{pa 2806 1006}%
\special{fp}%
%
\special{pn 8}%
\special{pa 2806 1006}%
\special{pa 2006 1406}%
\special{fp}%
%
\special{pn 8}%
\special{sh 1.000}%
\special{ar 2610 1410 50 50  0.0000000 6.2831853}%
\put(27.9500,-8.6500){\makebox(0,0){$u_i$}}%
\put(19.9500,-12.6500){\makebox(0,0){$u_{i-1}$}}%
\put(35.9500,-12.6500){\makebox(0,0){$u_{i+1}$}}%
%
\special{pn 8}%
\special{pa 2000 1400}%
\special{pa 2000 1600}%
\special{fp}%
%
\special{pn 8}%
\special{pa 3600 1600}%
\special{pa 3600 1400}%
\special{fp}%
%
\special{pn 8}%
\special{pa 3600 1600}%
\special{pa 3604 1634}%
\special{pa 3606 1668}%
\special{pa 3608 1702}%
\special{pa 3608 1734}%
\special{pa 3606 1764}%
\special{pa 3602 1794}%
\special{pa 3596 1820}%
\special{pa 3586 1844}%
\special{pa 3574 1866}%
\special{pa 3558 1886}%
\special{pa 3540 1904}%
\special{pa 3522 1920}%
\special{pa 3500 1934}%
\special{pa 3474 1946}%
\special{pa 3448 1958}%
\special{pa 3420 1966}%
\special{pa 3390 1974}%
\special{pa 3360 1982}%
\special{pa 3326 1986}%
\special{pa 3292 1992}%
\special{pa 3256 1994}%
\special{pa 3218 1998}%
\special{pa 3180 2000}%
\special{pa 3140 2000}%
\special{pa 3100 2002}%
\special{pa 3058 2002}%
\special{pa 3016 2002}%
\special{pa 2974 2002}%
\special{pa 2932 2002}%
\special{pa 2888 2000}%
\special{pa 2844 2000}%
\special{pa 2802 2000}%
\special{pa 2758 2000}%
\special{pa 2714 2000}%
\special{pa 2672 2002}%
\special{pa 2628 2002}%
\special{pa 2586 2002}%
\special{pa 2544 2002}%
\special{pa 2502 2002}%
\special{pa 2462 2000}%
\special{pa 2424 2000}%
\special{pa 2384 1998}%
\special{pa 2348 1996}%
\special{pa 2312 1992}%
\special{pa 2276 1988}%
\special{pa 2244 1982}%
\special{pa 2212 1974}%
\special{pa 2182 1966}%
\special{pa 2154 1958}%
\special{pa 2128 1946}%
\special{pa 2102 1934}%
\special{pa 2080 1920}%
\special{pa 2060 1904}%
\special{pa 2044 1886}%
\special{pa 2028 1866}%
\special{pa 2016 1846}%
\special{pa 2006 1822}%
\special{pa 2000 1794}%
\special{pa 1996 1766}%
\special{pa 1994 1736}%
\special{pa 1994 1704}%
\special{pa 1994 1670}%
\special{pa 1998 1636}%
\special{pa 2000 1602}%
\special{pa 2000 1600}%
\special{sp -0.045}%
%
\special{pn 8}%
\special{pa 2610 1410}%
\special{pa 2810 1010}%
\special{fp}%
\put(32.7000,-19.0000){\makebox(0,0){$b$}}%
%
\special{pn 8}%
\special{pa 2000 1400}%
\special{pa 2600 1400}%
\special{fp}%
%
\special{pn 13}%
\special{pa 2620 1410}%
\special{pa 2640 1438}%
\special{pa 2658 1464}%
\special{pa 2674 1492}%
\special{pa 2686 1522}%
\special{pa 2694 1552}%
\special{pa 2698 1584}%
\special{pa 2698 1616}%
\special{pa 2700 1650}%
\special{pa 2708 1680}%
\special{pa 2728 1704}%
\special{pa 2750 1728}%
\special{pa 2772 1754}%
\special{pa 2786 1780}%
\special{pa 2796 1812}%
\special{pa 2798 1844}%
\special{pa 2798 1876}%
\special{pa 2798 1908}%
\special{pa 2804 1938}%
\special{pa 2816 1968}%
\special{pa 2832 1996}%
\special{pa 2840 2010}%
\special{sp -0.045}%
%
\special{pn 8}%
\special{sh 1.000}%
\special{ar 3200 2000 40 40  0.0000000 6.2831853}%
%
\special{pn 8}%
\special{pa 3200 2000}%
\special{pa 2800 1000}%
\special{fp}%
%
\special{pn 8}%
\special{sh 1.000}%
\special{ar 2820 2000 40 40  0.0000000 6.2831853}%
\put(29.0000,-19.0000){\makebox(0,0){$v$}}%
\end{picture}%
\caption{The case when $\deg_G(u_i) = 4$ and $b \in V(C)$}
\label{t41-deg4-b}
\end{figure}

We next consider the case when neither $a$ nor $b$ lies on $C$.
Any four vertices in $\{u_{i-1},u_i,u_{i+1},a,b\}$ do not induce a $K_{1,3}$,
we have one of the following configurations by symmetry:
(1) $u_{i-1}u_{i+1}, ab \in E(G)$, (2) $u_{i-1}a, ab \in E(G)$ but $u_{i-1}u_{i+1} \notin E(G)$, 
and (3) $u_{i-1}a, bu_{i+1} \in E(G)$ (see Figure~\ref{t41-deg4}).
In the case (1), there is at most one of an inner $(a,u_{i-1})$-path and an inner $(b,u_{i+1})$-path,
and so we may assume that there is the former one (see the left of Figure~\ref{t41-deg4}).

\begin{figure}[htb]
\centering
\input{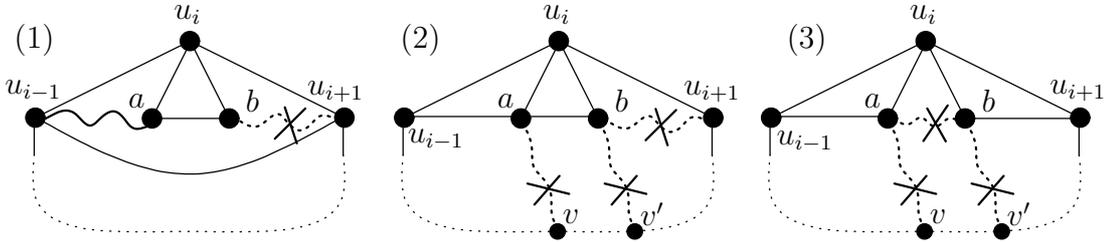}
\caption{The case when $\deg_G(u_i) = 4$}
\label{t41-deg4}
\end{figure}

By similar arguments in the case when $\deg_G(u_i) = 3$, as depicted in Figure~\ref{t41-deg4},
(1) there is no inner $(b,u_{i+1})$-path,
(2) there is neither an inner $(a,v)$-path, an inner $(b,v')$-path nor an inner $(b,u_{i+1})$-path,
(3) there is neither an inner $(a,v)$-path, an inner $(b,v')$-path nor an inner $(a,b)$-path,
where $v,v' \in V(C) \setminus \{ u_{i-1}, u_i, u_{i+1} \}$.
Thus, similarly to the above cases,
we can obtain another outerplanar embedding of $G$ with outer cycle longer than $C$,
by applying a ``jump" of the edge $u_{i-1}u_i$ over $a,b$ in cases (1) and (2) and over only $a$ in case (3).

Therefore, we can conclude that $G$ is outerplanar.

\medskip

Next, we show that $G$ is isomorphic to $\Zc_d, \Zc^*_d$ or a graph in $\Sc_t$.
If $|V(G)| \leq 4$, then we can easily verify that the theorem holds,
and hence, we assume that $|V(G)| \geq 5$.
Note that every vertex of degree at least~3 in a $K_{1,3}$-free graphs belongs to at least one triangle.
Thus, we divide the proof into the following two cases.

\medskip
\noindent
{\bf Case 1.} There is a pair of triangles sharing an edge.

Let $xuv$ and $uvy$ be two triangles sharing an edge $uv$.
If $\deg_G(u) = \deg_G(v) = 3$, then $|V(G)| \leq 4$ 
since $G$ is 2-connected and outerplanar, a contradiction.
Thus, at least one of $u$ and $v$ is of degree~4, say $u$.
Let $w_1$ be a neighbor of $u$ other than $v,x,y$.
Since $G$ is $K_{1,3}$-free, $w_1$ is adjacent to $y$ by symmetry.
If $\deg_G(v) = \deg_G(y) = 3$, then $G \cong \Zc_2$ and if there is a triangle $vyz$, then $G \in \Sc_3$.
Thus, $\deg_G(v) = 4$ or $\deg_G(y) = 4$ but there is no triangle sharing $vy$ with $uvy$.

We consider only the case when $\deg_G(v) = 3$ and $\deg_G(y) = 4$,
since two other cases (1) $\deg_G(v) = \deg_G(y) = 4$ and (2) $\deg_G(v) = 4$ and $\deg_G(y) = 3$ may be similarly proved.
Let $w_2$ be the fourth neighbor of $y$ (other than $u,v,w_1$).
Since $G$ is $K_{1,3}$-free and has no triangle sharing $vy$ with $uvy$, $G$ has a triangle $w_1w_2y$;
observe that $w_2$ can be adjacent to neither $x$ nor $v$ by the outerplanarity of $G$.
Then we next consider whether $\deg(w_1)=4$.
By repeating this argument, we can conclude that $G \cong \Zc_d$ or $G \cong \Zc^*_d$ for some $d \geq 2$.

\medskip
\noindent
{\bf Case 2.} Otherwise.

We may assume that any two triangles of $G$ do not share an edge.
In this case, we show that $G \in \Sc_t$ for some $t$.
Let $\partial G = u_0u_1 \dots u_{n-1}$.
Since it follows from what $G$ is $K_{1,3}$-free that any chord of $G$ is contained in a triangle, 
it suffices to show that every triangle of $G$ is $u_{i-1}u_iu_{i+1}$,
that is, the three vertices are consecutive on $\partial G$, where the subscripts are modulo $n$.
Suppose to the contrary that there is a triangle $u_iu_ju_k$ with $|j-i| \geq 2$, $j \neq i-2$ and $i < j < k$.
In this case, we have $\deg_G(u_i)=4$ or $\deg_G(u_j)=4$, and hence, 
without loss of generality, we suppose $\deg_G(u_i)=4$.
Similarly to the proof of Case~1, $u_j$ (or $u_k$) has to adjacent to $u_{i+1}$ (or $u_{i-1}$),
which contradicts that any two triangles do not share an edge.
Therefore, the three vertices of each triangle of $G$ are consecutive on $\partial G$,
which implies that $G \in \Sc_t$ for some $t$, since any two triangles share at most one vertex.
\end{proof}

By Theorem~\ref{thm:forbid_claw_k4}, we have the following corollary.

\begin{cor}\label{cor:claw_k4}
If a $2$-connected graph $G$ is $K_{1,3}$-free and $K_4$-minor-free,
then both $c_o(G)/c_e(G)$ and $c_e(G)/c_o(G)$ are bounded by a constant unless $G$ is a cycle.
\end{cor}
\begin{proof}
Let $G$ be a $2$-connected $K_{1,3}$-free and $K_4$-minor-free graph.
By Theorem~\ref{thm:forbid_claw_k4}, $G$ is isomorphic to $\Zc_d, \Zc^*_d$ or a graph in $\Sc_t$.
If $G$ is isomorphic to $\Zc_d$ or $\Zc^*_d$, 
then we are done by Theorem~\ref{thm:oddfaces} since every finite face of $G$ is triangular.
If $G \in \Sc_t$ and $G$ is not a cycle, then we are also done by Proposition~\ref{prop:dualstar}
since there is at least one triangular face corresponding to a leaf of a dual tree.
\end{proof}

As noted in the proof of Theorem~\ref{thm:forbid_claw_k4},
every vertex of degree at least~3 in a $K_{1,3}$-free graph belongs to a triangle.
Intuitively, 
if there is a cycle $C$ passing exactly one edge of a triangle $xyz$, say $xy$,
then there is another cycle $C'$ passing $xzy$ (ignoring the details whether $C$ passes $z$).
It is notable that the parity of the lengths of $C$ and that of $C'$ are different,
that is, we guess that 
the number of odd cycles in a 2-connected $K_{1,3}$-free graph is almost the same as that of even cycles.
Therefore, we conclude this section proposing the following challenging conjecture.
(Note that we can easily construct infinitely many non-2-connected $K_{1,3}$-free graphs $G$
such that $c_o(G)/c_e(G) \to \infty$ as $|V(G)| \to \infty$.)

\begin{conj}
For any $2$-connected $K_{1,3}$-free graph $G$, 
both $c_o(G)/c_e(G)$ and $c_e(G)/c_o(G)$ are bounded by a constant unless $G$ is a cycle.
\end{conj}

\section{Concluding remarks}\label{sec:remarks}

Throughout this paper, we addressed only 2-connected outerplanar graphs.
If an outerplanar graph $G$ is not 2-connected,
then $G$ consists of several blocks such that each two blocks share at most one vertex,
where a {\it block} of $G$ is a maximal 2-connected subgraph of $G$.
Note that each cycle in $G$ consists of edges in exactly one block.
Therefore, by applying our results to each block, 
we can obtain the corresponding results for any non-2-connected outerplanar graphs.
(We do not specifically write the statements of the corresponding results since it is just a tedious routine.)

For several 2-connected outerplanar graphs $G$ with prescribed dual trees,
a path, a star and a broom (the dual tree of $\Tc_n$),
we can show the results on the sharpness of the ratio of the numbers of odd and even cycles. 
It is a natural open problem to evaluate the ratio 
for 2-connected outerplanar graphs with other dual trees.
In particular, it is of interest to evaluate the ratio using the number of leaves of a dual tree.

For other particular graph classes, e.g., planar graphs, 
the analysis of the ratios seems more difficult and complicated 
even if the graph is a planar triangulation. 
Thus, we need to find some reasonable assumption or forbidden subgraphs/minors conditions for such graph classes,
and we believe that the study on this problem will give a huge contribution to the graph theory.

\end{document}